\documentclass[11pt,reqno]{amsart}
\usepackage{graphicx,amsmath,amsthm,verbatim,tikz,enumitem,hyperref, amssymb}
\hypersetup{colorlinks,linkcolor={red},citecolor={olive},urlcolor={red}}

\usepackage{mathtools}
\usepackage{xcolor}
\mathtoolsset{showonlyrefs}

\newcommand{\E}{\mathbf{E}}

\renewcommand{\P}{\mathbf{P}}
\newcommand{\f}{\frac}

\newcommand{\lrl}{\longleftrightarrow}
\newcommand{\dlrl}{\:\dot\longleftrightarrow\:}
\newcommand{\lrlf}{\overset{1}\longleftrightarrow}
\newcommand{\lrlb}{\overset{\h}{\longleftrightarrow}}
\newcommand{\Res}{\hat \bullet}

\newcommand{\la}{\longleftarrow}
\newcommand{\laf}{\overset{1}{\longleftarrow}}
\newcommand{\meet}{\textbf{\:---\:}}

\newcommand{\fir}{\overset{1}{\textbf{\:---\:}}}

\renewcommand{\emptyset}{\varnothing}
\renewcommand{\phi}{\varphi}

\newcommand{\cev}[1]{\reflectbox{\ensuremath{\vec{\reflectbox{\ensuremath{#1}}}}}}
\renewcommand{\b}{\bullet}
\renewcommand{\L}{\cev{\b}}
\newcommand{\R}{\vec{\b}}
\newcommand{\B}{\dot{\b}}
\newcommand{\h}{\hat{\b}}

\DeclareMathOperator{\rev}{rev}

\newcommand{\tsum}{\sum}

\usepackage{theoremref}

\newtheorem{theorem}{Theorem}

\newtheorem{proposition}[theorem]{Proposition}

\theoremstyle{definition}

\title[Arrivals are universal in coalescing ballistic annihilation]{Arrivals are universal in coalescing ballistic annihilation}
\author[Cruzado Padr\'o]{Dar\'io Cruzado Padr\'o}
\author[Junge]{Matthew Junge}
\author[Reeves]{Lily Reeves}

\thanks{All authors were partially supported by NSF grant DMS-1855516. Part of this research was completed during the 2021 Baruch College Discrete Math REU partially supported by NSF grant DMS-2051026}

\begin{document}

\maketitle

\begin{abstract}
Coalescing ballistic annihilation is an interacting particle system intended to model features of certain chemical reactions. Particles are placed with independent and identically distributed spacings on the real line and begin moving with velocities sampled from $-1,0,$ and $1$. Collisions result in either coalescence or mutual annihilation. For a variety of symmetric coalescing rules, we prove that the index of the first particle to arrive at the origin does not depend on the law for spacings between particles.
\end{abstract}


\section{Introduction}

We study a universality property of the coalescing ballistic annihilation process from \cite{benitez2020three}. These dynamics were introduced and studied by physicists at the end of the 20th century \cite{b3, b7,   blythe2000stochastic}. There has been a recent revival of interest from mathematicians \cite{HST, junge2018phase, burdinski, ST, bullets}.  The motivation for the dynamics comes from diffusion-limited annihilating systems \cite{bramson1991asymptotic} inspired by natural phenomena such as thermal variation \cite{toussaint1983particle}, turbulent
flows \cite{hill1976homogeneous}, and porous media \cite{raje2000experimental}. 

The process, though simple to define, has complex combinatorial and probabilistic structure.
As is commonly done, we consider the \emph{one sided} version of ballistic annihilation. The initial conditions have a particle $\b_k$ at $x_k \in (0,\infty)$ for $k \geq 1$ with interdistances $x_{k+1} - x_{k}$ that are independent and identically distributed with nonnegative, continuous probability distribution $\mu$.
Each particle is independently assigned a velocity from $-1,0,1$, and we designate the velocity of $\b_k$ with $\L_k, \B_k$, and $\R_k$, respectively. We assume that $\P(\B_k) =p$ and $\P(\L_k) = (1-p)/2 = \P(\R_k)$ for a fixed parameter $p \in (0,1)$.  At time $0$, each particle begins moving at its assigned velocity across $\mathbb R$. 

We denote two particles {$\b_{j}$ and $\b_{k}$} meeting at the same time and location by $\b_j \meet \b_k$.  Upon meeting the particles react. While there are many possible reactions, it was discovered in \cite{benitez2020three} that some are more amenable to analysis. We restrict our attention to \emph{three-parameter coalescing ballistic annihilation} (TCBA) systems. TCBA allows for moving particles to spontaneously survive collisions (equivalently, to generate a new moving particle), or to generate a $\B$-particle. 
Fix parameters $0\leq a,b,c \leq 1$ with $a+b\leq 1$. Let $[\b \meet \b \implies \Theta,\;\;  \theta]$
denote a collision that generates $\Theta \in \{ \B, \R, \L, \emptyset \}$ independently with probability $\theta$. The reaction rules are:

\begin{minipage}{.5\linewidth}
\begin{equation}
    \R \meet \L  \implies 
    \begin{cases}
        \L, &  a/2\\
        \R, &  a/2 \\
        \B, &  b \\
        \emptyset, & 1- (a+b)
    \end{cases}
\end{equation}
\end{minipage}%
\begin{minipage}{.5\linewidth}
\begin{align}
    \B \meet \L &\implies  
    \begin{cases} 
        \L, &  c\\
        \emptyset, & 1-c
    \end{cases}\\
    \R \meet \B &\implies  
        \begin{cases} 
            \R, &  c\\
            \emptyset,& 1-c
    \end{cases} .\label{eq:rules}
\end{align}
\end{minipage}
We will refer to the special case $a=b=c=0$ with only mutual annihilation as \emph{simple ballistic annihilation}.

Denote the event that the site $x$ is visited by the particle started at $x_k$ by $x \meet \L_k$. Let $A := A(\mu) = \min \{k \colon (0 \meet \L_k)\}$ be the index of the first particle to reach the origin. It was proven in \cite[Theorem 2]{HST} that the law of $A$ does not depend on the spacing distribution $\mu$ for simple ballistic annihilation. We generalize this to TCBA.

\begin{theorem}\thlabel{thm:main}
The law of $A$ does not depend on $\mu$ for TCBA and $\E[t^A]$ satisfies the recursion at \eqref{eq:f}.
\end{theorem}

One of the main quantities of interest in ballistic annihilation is $q := \P(A <\infty)$ and the phase transition $p_c = \sup \{ p \colon q =1\}$. It is proven in \cite{HST} for simple ballistic annihilation and in \cite{benitez2020three} for TCBA that $q$ and $p_c$ do not depend on $\mu$. In \cite{HST}, the authors further discovered that the \emph{skyline} of collision shapes for $p>p_c$ does not depend on $\mu$. In \cite{HT}, additional spacing universality properties were observed for simple ballistic annihilation as well as for an asymmetric version introduced in \cite{junge2018phase}. Broutin and Marckert proved that the related bullet process with finitely many particles has a universal law governing the number of surviving particles that does not depend on the velocity or spacing laws \cite{bullets2}.

Ballistic annihilation dynamics are notoriously sensitive to perturbation; changing the velocity of a single particle can have cascading effects. This feature makes the coalescing version significantly more complex. Our interest in establishing \thref{thm:main} comes from a desire to understand the limits of techniques successfully applied to simple ballistic annihilation. \thref{thm:main} marks a step in this direction and suggests that TCBA may share other universality properties with simple ballistic annihilation. An additional feature of \thref{thm:main} is the implicit recursion of the generating function of $A$. A special case of this recursion was utilized in \cite[Theorem 3]{HST} to describe the rate of decay of $\P(A>n)$. Our more general formula at \eqref{eq:f} is a first step towards describing the right tail of the distribution of $A$ in TCBA.

The method of proving \thref{thm:main} is similar to what was done in \cite{HST}. The idea is to prove by induction that the coefficients of the generating function $\E[t^A]$ do not depend on $\mu$. Coalescence makes the details more involved and requires additional considerations. 

\section{Proof of \thref{thm:main}}


We will write $\h$ to denote a stationary particle generated from a $\R \meet \L$ collision. If the collision involved $\R_j$ and $\L_k$, then we write $\h_{j,k}$ for the stationary particle now inhabiting $(x_j+x_k)/2$.


For positive integers $j$ and $k$, we define the collision events
\begin{align}
    \b_j \lrl \b_k &:= \{\text{$\b_j$ and $\b_k$ mutually annihilate}\} \\
    \b_j \lrlb \L_k &= \{\text{$\b_j$ and $\b_k$ mutually annihilate and generate $\h_{j,k}$}\} \\
    \b_j \la \L_k &:= \{\text{$\b_j$ is destroyed by $\L_k$ and $\L_k$ survives the collision}\}.
\end{align}
Note that in events such as $\b_j \la \L_k$, we continue to refer to the new particle generated as $\L_k$. 
Specify generic collision events by:
\begin{align}
	\b_j \lrl \L & := \{\text{there exists $k$ with $\{\b_j \lrl \L_k\}$}\}\\
    \R_j \lrl \B & := \{\text{there exists $k$ with $\{\R_j \lrl \B_k\}$}\}\\
    \R_j \lrl \h & := \{\text{there exist $k$ and $\ell$ with $\{\R_j \lrl \h_{k,\ell}\}$}\}.
\end{align}

To determine how reactions occur, we assign to each $\R$-particle a stack of independent instructions for $\R \meet \B$ reaction types with probabilities as at \eqref{eq:rules}. When $\R_j$ collides with a blockade the smallest index unused instruction is used to determine the reaction type. We assign to each $\L$-particle two independent stacks of reaction instructions distributed as at \eqref{eq:rules} to determine the outcomes of $\R \meet \L$ and $\b \meet \L$ collisions. This construction ensures that the reaction type of the next collision can be read off from the instructions before it occurs. Thus, the following visiting events are well-defined.
%
\begin{align}
    x_j \dlrl \b_k &:= (x_j \meet \b_k) \cap \{\b_k \text{ mutually annihilate in its next $\B$ collision} \} \\
    x_j \la \b_k &:= (x_j \meet \b_k) \cap \{\b_k \text{ survives its next $\B$ collision} \} \\
    x_j \fir\L_k &:= \{\text{$\L_k$ is the first left-moving particle to reach $x_j$}\}\\
     x_j \lrlf \L_k &:= (x_j \fir\L_k)\cap (x_j \lrl \L_k )\\
    x_j \laf \L_k &:= (x_j \fir\L_k)\cap (x_j \la \L_k ).
\end{align}

Note that we will use the symbols $\cap$ and $\wedge$ interchangeably. Given an interval $I$ and an event $B$, we write $B_I$ for the event in TCBA restricted to only the particles initially in $I$. We will use various forms of renewal that occur in TCBA. These come from the fact that the particles behind a moving particle cannot influence events involving the moving particle. For example, $\P(( x_\ell \meet \L ) \mid (\B_1 \lrl \L_\ell) ) =  \P(x_\ell \meet \L) = \P(0 \meet \L).$

Our main tool is the following decomposition result for $\P(A=n)$. 

\begin{proposition} \thlabel{lem:pn}
Let $p_n:= \P(A=n)$. For $n \geq 2$, it holds that 
\begin{align}
	p_n=\alpha_n + \dot \beta_n + \hat \beta_n + \gamma_n  + \hat \gamma_n + \cev \gamma_n \label{eq:pn}
\end{align}
with
\begin{align}
    \alpha_n &:= \P[(A=n) \wedge \B_1] = cpp_{n-1} + (1-c)p\tsum_{1<k<n} p_{k-1}p_{n-k} \label{eq:a} \\
    \dot \beta_n &:= \P[(A=n) \wedge (\R_1 \lrl \B )]=  \tfrac{1-(a+b)}2p \tsum_{1<k<n} p_{k-1}p_{n-k} \label{eq:db}\\
    \hat\beta_{n} &:= \P[(A=n) \wedge (\R_1 \lrl \Res )]=\tfrac{1-(a+b)}2 \tsum_{1<k<\ell<n}  \hat \delta_{\ell-k+1}p_{k-1} p_{n-\ell}\label{eq:hb}\\
    \gamma_n &:= \P[(A=n) \wedge (\R_1 \lrl \L )] = \tsum_{1<k<n} \delta_{k} p_{n-k}\label{eq:g}\\
    \hat \gamma_n &:= \P[(A=n) \wedge (\R_1 \lrlb \L )] \\
                &\hspace{ 2.5 cm} =  c \tsum_{1<k<n} \hat \delta_{k}p_{n-k} +  (1-c)\tsum_{1<k<\ell<n} \hat \delta_{k} p_{\ell-k}p_{n-\ell} \label{eq:hg}\\
    \cev\gamma_n &:= \P[(A=n) \wedge (\R_1 \la \L )] = \tfrac a2 \bar \delta_n\label{eq:cg}\\
    \bar \delta_n &:= \P(\R_1 \meet \L_n) \\
        &\hspace{ 1.5 cm }= \tfrac{1-p}{2}p_{n-1} - (1-c) \tsum_{1<k<n} [\dot \beta_k + \hat \beta_k + \tfrac 12 c p_{k-1}] p_{n-k} \\
    & \hspace{ 4.5 cm} -  \tfrac 12 (1-c) c \sum_{1<k<\ell<n} \hat \delta_{\ell-k+1} p_{k-1} p_{n-\ell}. \label{eq:bd} \\
    \delta_n &:= \P(\R_1 \lrl \L_n) = (1-(a+b)) \bar \delta_n \label{eq:d} \\
    \hat \delta_n &:= \P( \R_1 \lrlb \L_n) = b \bar \delta_n.\label{eq:hd}
\end{align}
\end{proposition}

\begin{proof}[Proof of \eqref{eq:pn}]
This is a partitioning of the event $\{A=n\}$ based on the velocity of $\b_1$. We use the fact (ensured by symmetry) that $\R_1$ is almost surely annihilated as observed in \cite{benitez2020three}.
\end{proof}

\begin{proof}[Proof of \eqref{eq:a}]
Conditional on $\B_1$, there are precisely two manners in which $A = n$. One is that $\L_n$ is the first left-moving particle to reach $x_1$ and the reaction $\B-\L\implies \L$ occurs. This occurs with probability 
\begin{align}
\P(\B_1) \P( \B \meet \L \implies \L)p_{n-1} = p c p_{n-1}.\label{eq:pass}	
\end{align}
 The other manner in which $A=n$ may occur conditional on $\B_1$, is if there is some $1 < k < n$ such that $\B_k$ is the first particle to reach $x_1$ from the right and a $[\B \meet \L \implies \varnothing]$ reaction occurs. Then $\B_n$ is the first to reach $x_k$ from the right. This second part happens with probability $\P( A = k - 1)\P( A = n - k ) = p_{k-1}p_{n-k}$. So, for each $k$ we acquire the probability 
 $$\P(\B_1) \P(\B \meet \L \implies \varnothing) p_{k-1}p_{n-k} = p (1-c) p_{k-1} p_{n-k}.$$
 Summing over $k$ and combining with \eqref{eq:pass} gives \eqref{eq:a}. 
\end{proof}

\begin{figure}
        \begin{tikzpicture}[scale =  1.5]
			\draw (0,0)--(7,0);
			\coordinate(origin) at (0,0);
			\coordinate(O) at (0,0);
			\coordinate(p1) at (.5,.15);
			\coordinate(P1) at (0.5,0.3);
			\coordinate(pn) at (6.5,.15);
			\coordinate(Pn) at (6.35,.3);
			\coordinate(pk) at (3.1,.15);
			\coordinate(Pk1) at (3,.3);
			\coordinate(Pk2) at (3.1,.3);
			\node[below] at (O) {\small 0 };
			\node at (pn) {$\L_n$};
			\node at (p1) {$\R_1$};
			\node at (pk) {$\B_k$};
			\draw[->] (P1) to [bend left = 40] (Pk1);
			\draw[->] (Pn) to [bend right = 40] (Pk2);
			\draw[->] (Pn) to [bend right = 40] (Pk2);
			\draw[|-|] (.5,-.2) -- ++(2.5,0) node [midway,fill=white] {\small $x_k -x_1$};
            \draw[|-|] (3,-.4) -- ++(3.5,0) node [midway,fill=white] {\small $x_n -x_k$};
		\end{tikzpicture}
		
		\vspace{1 cm}
		
	\begin{tikzpicture}[scale =  1.5]
			\draw (0,0)--(7,0);
			\coordinate(origin) at (0,0);
			\coordinate(O) at (0,0);
			\coordinate(p1) at (.5,.15);
			\coordinate(P1) at (0.5,0.3);
			\coordinate(pn) at (6.5,.15);
			\coordinate(Pn) at (6.35,.3);
			\coordinate(pk) at (4,.15);
			\coordinate(Pk1) at (3.85,.3);
			\coordinate(Pk2) at (3.95,.3);
			\node[below] at (O) {\small 0 };
			\node at (pn) {$\L_n$};
			\node at (p1) {$\R_1$};
			\node at (pk) {$\B_{k'}$};
			\draw[->] (P1) to [bend left = 40] (Pk1);
			\draw[->] (Pn) to [bend right = 40] (Pk2);
			\draw[|-|] (.5,-.2) -- ++(3.5,0) node [midway,fill=white] {\small $x_{k'} -x_1$};
            \draw[|-|] (4,-.4) -- ++(2.5,0) node [midway,fill=white] {\small $x_n -x_{k'}$};
		\end{tikzpicture}
    \caption{A configuration $\omega \in \{(A=n) \wedge (\R_1 \lrl \B_k)\}$ (top) and its reversal (bottom). Particles between $\b_1$ and $\b_k$ and $\b_k$ and $\b_n$ are not shown. The arced arrows indicate that the particle is the first to reach that site among all particles started under the arc. Summing over all $k$ and $k'$ gives complementary events. This allows us to bypass any computations involving the interdistances.}
    \label{fig:beta}
\end{figure}
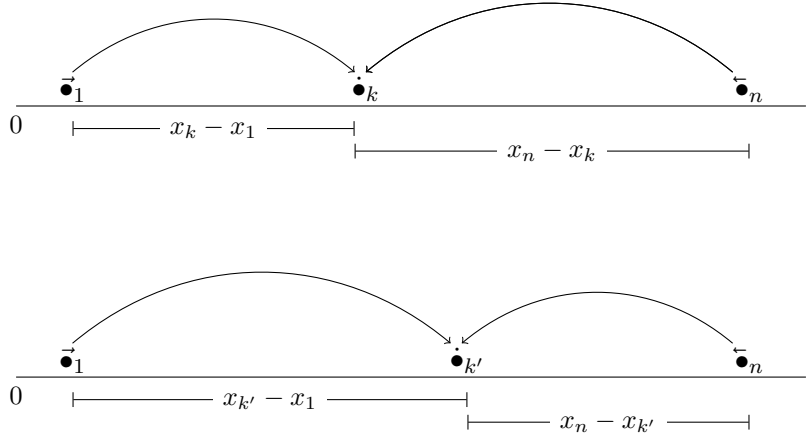

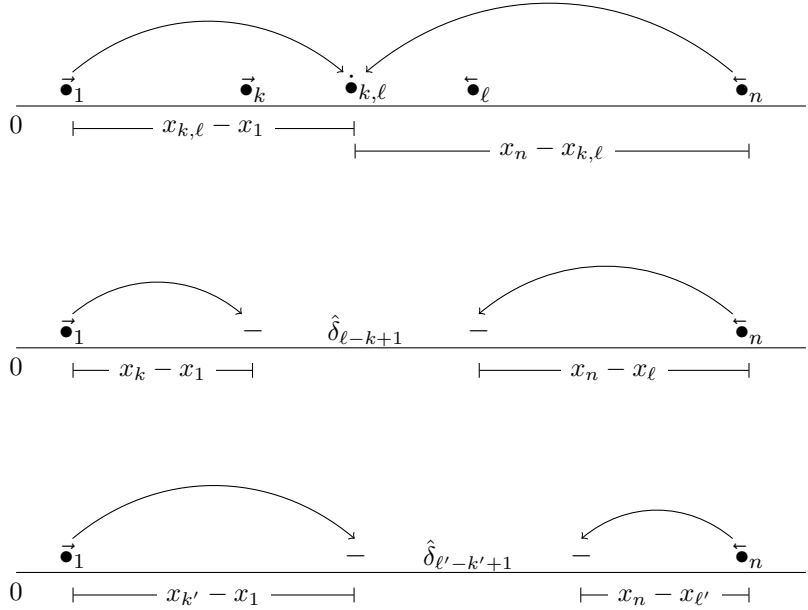
\begin{figure}
    	\begin{tikzpicture}[scale =  1.5]
			\draw (0,0)--(7,0);
			\coordinate(origin) at (0,0);
			\coordinate(O) at (0,0);
			\coordinate(p1) at (.5,.15);
			\coordinate(P1) at (0.5,0.3);
			\coordinate(pn) at (6.5,.15);
			\coordinate(Pn) at (6.35,.3);
			\coordinate(pk) at (2.1,.15);
			\coordinate(Pk1) at (2,.3);
			\coordinate(Pk2) at (2.1,.3);
                \coordinate(pl) at (4.1,.15);
			\coordinate(Pl1) at (4,.3);
			\coordinate(Pl2) at (4.1,.3);
                \coordinate(pkl) at (3.1,.15);
			\coordinate(Pkl1) at (2.9,.3);
			\coordinate(Pkl2) at (3.1,.3);
			\node[below] at (O) {\small 0 };
			\node at (pn) {$\L_n$};
			\node at (p1) {$\R_1$};
			\node at (pk) {$\R_k$};
                \node at (pl) {$\L_\ell$};
                \node at (pkl) {$\B_{k,\ell}$};
			\draw[->] (P1) to [bend left = 40] (Pkl1);
			\draw[->] (Pn) to [bend right = 40] (Pkl2);
			\draw[|-|] (.5,-.2) -- ++(2.5,0) node [midway,fill=white] {\small $x_{k,\ell} -x_1$};
            \draw[|-|] (6.5,-.4) -- ++(-3.5,0) node [midway,fill=white] {\small $x_n -x_{k,\ell}$};
		\end{tikzpicture}
		
		\vspace{1 cm}

            \begin{tikzpicture}[scale =  1.5]
			\draw (0,0)--(7,0);
			\coordinate(origin) at (0,0);
			\coordinate(O) at (0,0);
			\coordinate(p1) at (.5,.15);
			\coordinate(P1) at (0.5,0.3);
			\coordinate(pn) at (6.5,.15);
			\coordinate(Pn) at (6.35,.3);
			\coordinate(pk) at (2.1,.15);
			\coordinate(Pk1) at (2,.3);
			\coordinate(Pk2) at (2.1,.3);
                \coordinate(pl) at (4.1,.15);
			\coordinate(Pl1) at (4,.3);
			\coordinate(Pl2) at (4.1,.3);
                \coordinate(pkl) at (3.1,.15);
			\coordinate(Pkl1) at (2.9,.3);
			\coordinate(Pkl2) at (3.1,.3);
			\node[below] at (O) {\small 0 };
			\node at (pn) {$\L_n$};
			\node at (p1) {$\R_1$};
			\node at (pk) {$-$};
                \node at (pl) {$-$};
                \node at (pkl)  {\small $\hat \delta_{\ell-k+1}$};
			\draw[->] (P1) to [bend left = 40] (Pk1);
			\draw[->] (Pn) to [bend right = 40] (Pl2);
			\draw[|-|] (.5,-.2) -- ++(1.6,0) node [midway,fill=white] {\small $x_k -x_1$};
            \draw[|-|] (6.5,-.2) -- ++(-2.4,0) node [midway,fill=white] {\small $x_n -x_\ell$};
		\end{tikzpicture}
		
	\vspace{1 cm}

            \begin{tikzpicture}[scale =  1.5]
			\draw (0,0)--(7,0);
			\coordinate(origin) at (0,0);
			\coordinate(O) at (0,0);
			\coordinate(p1) at (.5,.15);
			\coordinate(P1) at (0.5,0.3);
			\coordinate(pn) at (6.5,.15);
			\coordinate(Pn) at (6.35,.3);
			\coordinate(pk) at (3.01,.15);
			\coordinate(Pk1) at (3.0,.3);
			\coordinate(Pk2) at (3.01,.3);
                \coordinate(pl) at (5.01,.15);
			\coordinate(Pl1) at (5.0,.3);
			\coordinate(Pl2) at (5.01,.3);
                \coordinate(pkl) at (4.01,.15);
			\coordinate(Pkl1) at (3.4,.3);
			\coordinate(Pkl2) at (3.51,.3);
			\node[below] at (O) {\small 0 };
			\node at (pn) {$\L_n$};
			\node at (p1) {$\R_1$};
			\node at (pk) {$-$};
                \node at (pl) {$-$};
                \node at (pkl)  {\small $\hat \delta_{\ell'-k'+1}$};
			\draw[->] (P1) to [bend left = 40] (Pk1);
			\draw[->] (Pn) to [bend right = 40] (Pl2);
			\draw[|-|] (.5,-.2) -- ++(2.5,0) node [midway,fill=white] {\small $x_{k'} -x_1$};
            \draw[|-|] (6.5,-.2) -- ++(-1.5,0) node [midway,fill=white] {\small $x_n -x_{\ell'}$};
		\end{tikzpicture}
    \caption{A configuration $\omega \in \{(A=n) \wedge (\R_1 \lrl \B_{k,\ell})\}$ (top). The middle figure shows an equivalent formulation conditional on $\R_k \lrlb \L_\ell$. The bottom figure shows the configuration after reversing the particles in $[x_1,x_k)$. }
    \label{fig:betah}
\end{figure}
\begin{proof}[Proof of \eqref{eq:db}]
The event $\{(A = n) \wedge (\R_1 \lrl \B_k)\}$ occurs if and only if the following hold:
\begin{itemize}
    \item  $\B_k$ occurs. 
    \item The
first particle to reach $x_k$ from the left is $\R_1$, which mutually annihilates with $\B_k$. 
    \item The first particle to reach $x_k$ from the right is $\L_n$. 
    \item And, $x_k - x_1 <x_n -x_k$. 
\end{itemize}
 
We can then write
\begin{align}
\dot \beta_n &= \textstyle \sum_{1 < k < n} \P\bigl( (\B_k) \wedge (\R_1 \lrlf x_k)_{(0,x_k)} \wedge (x_k \fir \L_n)_{(x_k,\infty)}\\
&\hspace{7.5 cm } \wedge (x_k - x_1 < x_n - x_k)\bigr). \label{eq:db1}
\end{align}

Given a configuration $\omega \in \{(A=n) \wedge (\R_1 \lrl \B_k)\}$ of particle locations and velocities, define $\rev_n(\omega)$ to be the reversed configuration. So, the particle at $x \in [x_1,x_n]$ corresponds to the particle in $\rev_n(\omega)$ with position $x_1 + (x_n - x)$ and is moving in the opposite direction. Symmetry of the parameters ensures that reversing the configuration preserves the probability: $\P(\omega) = \P(\rev_n(\omega))$.
Since reversing maps the index $k$ to $k' = n+1 -k$, we may also write
\begin{align}
\dot \beta_n&:= \textstyle\sum_{1 < k' < n} \P\bigl( (\B_{k'}) \wedge (\R_1 \lrlf x_{k'})_{(0,x_{k'})} \wedge (x_{k'} \fir \L_n)_{(x_{k'},\infty)}\\
&\hspace{7.5 cm} \wedge (x_{k'} - x_1 > x_n - x_{k'})\bigr). \label{eq:db2}
\end{align} 

Summing the two formulas for $\dot \beta_n$ and combining terms with the same index  partitions the comparison between $x_k-x_1$ and $x_n-x_k$. See Figure~\ref{fig:beta}. Thus,
\begin{align}
2 \dot \beta_n &= \tsum_{1 < k < n} \P( (\B_k) \wedge (\R_1 \lrlf x_k)_{(0,x_k)} \wedge (x_k \fir \L_n)_{(x_k,\infty)} )	\\
               &= \tsum_{1<k<n} p (1-(a+b)) p_{k-1} p_{n-k}.
\end{align}
At the second step we apply independence. Dividing by $2$ gives the claimed formula.
\end{proof}

\begin{proof}[Proof of \eqref{eq:hb}]
 The event $\{(A = n) \wedge (\R_1 \lrl \h_{k,\ell})\}$ occurs if and only if the following hold:
 \begin{itemize}
    \item $\h_{k,\ell}$ is generated from $\R_k \lrlb \L_\ell$ at $x_{k,\ell} = (x_k + x_\ell)/2$ for some $1<k <\ell<n$. \item The
first particle to reach $x_{k,\ell}$ from the left of $x_k$ is $\R_1$, which mutually annihilates with $\B_{k,\ell}$. 
    \item The first particle to reach $x_{k,\ell}$ from the right of $x_\ell$ is $\L_n$. 
    \item And, $x_{k,\ell} - x_1 <x_n -x_{k,\ell}$. 
\end{itemize}
    Thus,
\begin{align}
\hat \beta_n &= \sum_{1 < k < \ell < n} \P\bigl( (\R_1 \lrlf x_{k})_{(0,x_k)} \wedge (\R_k \lrlb \L_\ell)_{[x_k,x_\ell]} \\
&\hspace{5 cm }\wedge (x_\ell \fir \L_n)_{(x_\ell,\infty)} \wedge (x_{k,\ell} - x_1 < x_n - x_{k,\ell})\bigr). \label{eq:dhb1}
\end{align}
Let $G_{k,\ell} = (\R_k \lrlb \L_\ell)_{[x_k,x_\ell]}$ so that  $\P(G_{k,\ell}) = \hat \delta_{\ell-k+1}$. Conditioning gives
\begin{align}
\hat \beta_n &=  \sum_{1 < k < \ell < n} \hat \delta_{\ell-k+1} \P\bigl( (\R_1 \lrlf x_{k})_{(0,x_k)} \wedge (x_\ell \fir \L_n)_{(x_\ell,\infty)}  \\
&\hspace{6 cm } \wedge (x_{k,\ell} - x_1 < x_n - x_{k,\ell}) \mid G_{k,\ell}\bigr). \label{eq:dhb2}
\end{align}

Since moving particles have unit speed, we have $x_{k,\ell} - x_1 < x_n - x_{k,\ell}$ if and only if $x_k -x_1 < x_n - x_\ell$. Using this observation and the fact that the events $(\R_1 \lrlf x_{k})_{(0,x_k)}$ and $(x_\ell \fir \L_n)_{(x_\ell,\infty)}$ are independent of $G_{k,\ell}$ yields
\begin{align}
\hat \beta_n &=  \sum_{1 < k < \ell < n} \hat \delta_{\ell-k+1} \P\bigl( (\R_1 \lrlf x_{k})_{[x_1,x_k)} \wedge (x_\ell \fir \L_n)_{(x_\ell,\infty)}  \\
&\hspace{6.5 cm } \wedge (x_{k} - x_1 < x_n - x_{\ell})\Bigr ). \label{eq:dhb3}
\end{align}

By reversing the configuration of particles in $[x_1,x_n]$ as with the proof of \eqref{eq:db} (illustrated in Figure~\ref{fig:betah}), we may also write 
\begin{align}
\hat \beta_n &= \sum_{1 < k' < \ell' < n} \hat \delta_{\ell'-k'+1} \P\bigl( (\R_1 \fir x_{k'})_{[x_1,x_{k'})} \wedge (x_{\ell'} \lrlf \L_n)_{(x_{\ell'},\infty)}  \\
&\hspace{6.5 cm } \wedge (x_{k'} - x_1 > x_n - x_{\ell'})\Bigr ). \label{eq:dhb4}
\end{align}
Since reactions are determined independently, we can swap the reaction types in \eqref{eq:dhb4} so that the probabilities 
$$\P\bigl( (\R_1 \fir x_{k'})_{[x_1,x_{k'})} \wedge (x_{\ell'} \lrlf \L_n)_{(x_{\ell'},\infty)} \wedge (x_{k'} - x_1 > x_n - x_{\ell'})\Bigr )$$
and 
$$\P\bigl( (\R_1 \lrlf x_{k'})_{[x_1,x_{k'})} \wedge (x_{\ell'} \fir \L_n)_{(x_{\ell'},\infty)} \wedge (x_{k'} - x_1 > x_n - x_{\ell'})\Bigr )$$
are equal. We may then rewrite \eqref{eq:dhb4} as
\begin{align}
\hat \beta_n &= \sum_{1 < k' < \ell' < n} \hat \delta_{\ell'-k'+1} \P\bigl( (\R_1 \lrlf x_{k'})_{[x_1,x_{k'})} \wedge (x_{\ell'} \fir \L_n)_{(x_{\ell'},\infty)}  \\
&\hspace{6.5 cm } \wedge (x_{k'} - x_1 > x_n - x_{\ell'})\Bigr ). \label{eq:dhb5}
\end{align}

Summing the two formulations of $\hat \beta$ at \eqref{eq:dhb3} and \eqref{eq:dhb5} removes the interval comparisons. Thus,
\begin{align}
2\hat \beta_n &=  \sum_{1 < k < \ell < n} \hat \delta_{\ell-k+1} \P\Bigl( (\R_1 \lrlf x_{k})_{[x_1,x_{k})} \wedge (x_{\ell} \fir \L_n)_{(x_{\ell},\infty)}\Bigr)\\
&=  [1-(a+b)]\sum_{1 < k < \ell < n} \hat \delta_{\ell-k+1} p_{k-1} p_{n-\ell}.
\end{align}
Dividing by $2$ gives \eqref{eq:hb}.
\end{proof}

\begin{proof}[Proof of \eqref{eq:g}]
Taking the definition of $\delta_n$ for granted, it is straightforward to see that
\begin{align}
	\P(\gamma_n) = \tsum_{1 < k < n} \P(\R_1 \lrl \L_k) \P(x_k \fir \L_n) &= \tsum_{1 < k < n} \delta_k p_{n-k}
\end{align}
which gives \eqref{eq:g}. 
\end{proof}
\begin{proof}[Proof of \eqref{eq:hg}]
The event in the probability at \eqref{eq:hg} may occur in two ways. One, there exists a $1<k<n$ such that:
$(\R_1 \lrlb \L_k) \wedge (\h_{1,k} \la \L_n)$ occurs. Each such event is equivalent to $(\R_1 \lrlb \L_k) \wedge (x_k \la \L_n)$, which has probability $\hat \delta_k c p_{n-k}$. 
%
The other manner in which the event in the probability at \eqref{eq:hg} may occur is if for $1<k<\ell<n$ we have 
$$(\R_1 \lrlb \L_k) \wedge (x_{1,k} \lrlf \L_\ell) \wedge (x_\ell \fir \L_n)).$$ Conditional independence ensures that this event has probability $\hat \delta_k (1-c) p_{\ell-k+1} p_{n-\ell}$
as claimed in the second part of \eqref{eq:hg}. 
\end{proof}

\begin{proof}[Proof of \eqref{eq:cg}]
The formula for $\cev \gamma_n$ is the simple observation that the event in question occurs if and only if $\{\R_1 \la \L_n\}$, which has the claimed probability. 
\end{proof}

\begin{proof}[Proof of \eqref{eq:bd}, \eqref{eq:d}, and \eqref{eq:hd}]
The main work is proving \eqref{eq:bd}. The other two formulas follow immediately from specifying the reaction. Towards \eqref{eq:bd}, let 
$$G = \R_1 \wedge (x_1 \fir \L_n)_{(x_1,\infty)}.$$
We can easily compute $\P(G) = \f{1-p}2 p_{n-1}.$
We further claim that
\begin{align}
 (\R_1 \meet \L_n) = G \setminus [\dot B_1 \cup \dot B_2 \cup  \hat B_1 \cup \hat B_2] \label{eq:d1}
\end{align}
with 
\begin{align}
	\dot B_1 &=  \bigcup_{1<k<n} (x_1 \lrlf \L_k)_{(x_1,x_k]} \wedge (\R_1 \lrl \B) \wedge (x_k \fir \L_n)_{(x_k,x_n]}\\
        \dot B_2 & = \bigcup_{1<k<n} (\R_1 \lrl \B_k)_{[x_1,x_k]} \wedge (x_k \laf \L_n)_{(x_k,x_n]} \wedge (x_k - x_1 < x_n - x_k) 
\end{align}
and
\begin{align}
	\hat B_1 &= \bigcup_{1<k<n} (x_1 \lrlf \L_k)_{(x_1,x_k]} \wedge (\R_1 \lrl \h) \wedge (x_k \fir \L_n)_{(x_k,x_n]}\\
    \hat B_2 &= \bigcup_{1 < k < \ell < n} (\R_1 \lrlf x_k)_{[x_1,x_{k})} \wedge (\R_k \lrlb \L_\ell)_{[x_k,x_\ell]} \wedge (x_\ell \laf \L_n)_{(x_\ell,x_n]}.
\end{align}

To see why \eqref{eq:d1} holds, first note that $G$ is necessary for $\R_1 \meet \L_n$. 
Next, we claim that $\dot B_1 \cup \dot B_2 \cup  \hat B_1 \cup \hat B_2$ contains precisely the configurations in $G$ for which $\R_1$ does not collide with $\L_n$. 
Indeed, $\R_1$ cannot collide with a smaller index $\L$-particle, since otherwise, that smaller index $\L$-particle would reach $x_1$ before $\L_n$ in the process restricted to $(x_1,x_n)$. 
So, the configurations from $G$ for which $\R_1$ does not collide with $\L_n$ must have $\R_1$ mutually annihilating with a blockade. 
The events in $\dot B_1 \cup \dot B_2$ describe the configurations for which $\R_1 \lrl \B$ and $x_1 \fir \L_n$.  See Figure~\ref{fig:delta}.
The configurations  in $\hat B_1 \cup \hat B_2$ describe the configurations for which $\R_1 \lrl \h$ and $x_1 \fir \L_n$.

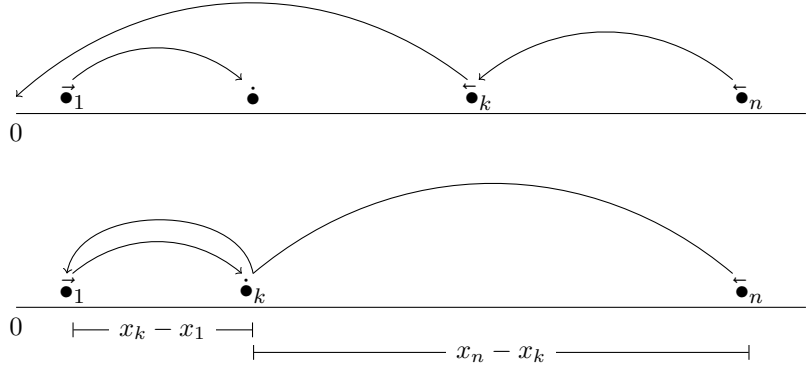
\begin{figure}
    	\begin{tikzpicture}[scale =  1.5]
			\draw (0,0)--(7,0);
			\coordinate(origin) at (0,0);
			\coordinate(O) at (0,0);
                \coordinate(O1) at (0,.15);
			\coordinate(p1) at (.5,.15);
			\coordinate(P1) at (0.5,0.3);
			\coordinate(pn) at (6.5,.15);
			\coordinate(Pn) at (6.35,.3);
			\coordinate(pk) at (2.1,.15);
			\coordinate(Pk1) at (2,.3);
			\coordinate(Pk2) at (2.1,.3);
                \coordinate(pl) at (4.1,.15);
			\coordinate(Pl1) at (4,.3);
			\coordinate(Pl2) at (4.1,.3);
                \coordinate(pkl) at (3.1,.15);
			\coordinate(Pkl1) at (2.9,.3);
			\coordinate(Pkl2) at (3.1,.3);
			\node[below] at (O) {\small 0 };
			\node at (pn) {$\L_n$};
			\node at (p1) {$\R_1$};
			\node at (pk) {$\B$};
                \node at (pl) {$\L_k$};
			\draw[->] (P1) to [bend left = 40] (Pk1);
            \draw[->] (Pl1) to [bend right = 40] (O1);
			\draw[->] (Pn) to [bend right = 40] (Pl2);
		\end{tikzpicture}
  \begin{tikzpicture}[scale =  1.5]
			\draw (0,0)--(7,0);
			\coordinate(origin) at (0,0);
			\coordinate(O) at (0,0);
                \coordinate(O1) at (0,.15);
			\coordinate(p1) at (.5,.15);
			\coordinate(P1) at (0.5,0.3);
                \coordinate(P2) at (0.45,0.3);
			\coordinate(pn) at (6.5,.15);
			\coordinate(Pn) at (6.35,.3);
			\coordinate(pk) at (2.1,.15);
			\coordinate(Pk1) at (2,.3);
			\coordinate(Pk2) at (2.1,.3);
                \coordinate(pl) at (4.1,.15);
			\coordinate(Pl1) at (4,.3);
			\coordinate(Pl2) at (4.1,.3);
                \coordinate(pkl) at (3.1,.15);
			\coordinate(Pkl1) at (2.9,.3);
			\coordinate(Pkl2) at (3.1,.3);
			\node[below] at (O) {\small 0 };
			\node at (pn) {$\L_n$};
			\node at (p1) {$\R_1$};
			\node at (pk) {$\B_k$};
			\draw[->] (P1) to [bend left = 40] (Pk1);
			\draw (Pn) to [bend right = 40] (Pk2);
                \draw[->] (Pk2) to [bend right = 80] (P2);
			\draw[|-|] (.5,-.2) -- ++(1.6,0) node [midway,fill=white] {\small $x_{k} -x_1$};
            \draw[|-|] (6.5,-.4) -- ++(-4.4,0) node [midway,fill=white] {\small $x_n -x_{k}$};
		\end{tikzpicture}
\caption{The top diagram shows a configuration in $\R_1 \wedge (x_1 \fir \L_n)_{(x_1,x_n]}$ for which $\R_1 \meet \L_n$ fails to occur. Arrows indicate that the particle from the tail of the arrow is the first to visit the location at the head of the arrow. The bottom diagram shows another type of configuration for which this may occur. Note that $\L_n$ survives the indicated $\B \meet \L_n$ collision.} \label{fig:delta}
\end{figure}

Using independence and the definition of $\dot \beta_k$, it is easily seen that
\begin{align}
    \P(\dot B_1) &= \tsum_{1<k < n} (1-c) \dot \beta_k p_{n-k}.
\end{align}
    A similar reversal argument as in the proof of \eqref{eq:db} yields
\begin{align}
    \P(\dot B_2) &= \tsum_{1<k<n } \tfrac 12 (1-c)c p_{k-1} p_{n-k}.
\end{align}
     Similarly, 
\begin{align}
    \P(\hat B_1) &= \tsum_{1<k<n } (1-c)\hat \beta_k p_{n-k},
\end{align}
     and a a reversal argument like the one used to obtain \eqref{eq:hb} yields 
\begin{align}
    \P(\hat B_2) &=  \tfrac 12  (1-c) c \sum_{k < \ell < n} \hat \delta_{\ell- k+1} p_{k-1} p_{n-\ell}.
\end{align}
 Since all the individual events in $\dot B_1 \cup \dot B_2 \cup  \hat B_1 \cup \hat B_2$ are disjoint, we subtract these equations from \eqref{eq:d1} to obtain \eqref{eq:bd}. 
\end{proof}

\begin{proof}[Proof of \thref{thm:main}]
For $n=1$, it is immediate that $p_1= \P(\L_1) = (1-p)/2$ and the quantities in \eqref{eq:a}--\eqref{eq:hd} are all equal to $0$. Hence, none depend on $\mu$. Given $n \geq 2$, it follows from \thref{lem:pn} that the quantities in \eqref{eq:a}--\eqref{eq:hd} can be expressed soley in terms of version of the quantities with index strictly less than $n$. Thus, we may proceed by induction to infer that these quantities do not depend on $\mu$ for all $n$. It follows then from \eqref{eq:pn} that  $p_n$ does not depend on $\mu$. Since $f(t)= \E[t^A] = \sum_{n\geq 1} p_n t^n$
uniquely determines the distribution of $A$, we obtain the first part of \thref{thm:main}.

The implicit recursion for $f$ is obtained by summing the generating functions corresponding to both sides of \eqref{eq:pn} and then applying the equations \eqref{eq:a}--\eqref{eq:hd}. This gives
\begin{align}
	A(t) &:= \tsum_{n\geq 0} \alpha_nt^n = c  p t f(t) + (1-c) p t f(t)^2 \\
	B(t) &:= \tsum_{n \geq 0} (\beta_n + \hat \beta_n)t^n = \f{1-(a+b)}2 p  f(t)^2 + \f 12  \hat D(t) f(t)^2 \\
	C(t) &:= \tsum_{n \geq 0} (\gamma_n + \hat \gamma_n + \cev \gamma_n)t^n\\
        &= D(t) f(t) + c \hat D(t) f(t)  + (1- c)\hat D(t) f(t)^2+ \f{a/2}{1-(a+b) } D(t)\\
        \bar D(t) &= \tsum_{n \geq 0} \bar \delta_n t^n= \tfrac{1-p} 2 t f(t) - (1-c) f(t) \bigl[B(t) 
+ \tfrac 12 c t  - \tfrac 12  c \hat D(t) f(t)\bigr] \\
	D(t) &:= \tsum_{n \geq 0} \delta_nt^n = ( 1- (a+b) ) \bar D(t)\\
	\hat D(t) &:= \tsum_{n \geq 0} \hat \delta_n t^n = b \bar D(t).
\end{align}

As an example of the calculations that lead to the formulas for $A,B,C,D,$ and $\hat D$, we provide the derivation for summing the $\hat \beta_n$.
First, we apply the formula at \eqref{eq:hb} to write
\begin{align}
    \sum_{n=0}^\infty \hat \beta_n t^n &= \tfrac{1-(a+b)}{2} \sum_{n=0}^\infty \sum_{1 < k < \ell <n} \hat \delta_{\ell - k +1}p_{k-1} p_{n -\ell} t^n. \label{eq:hbs}
\end{align}
Expanding and rearranging the sums, then applying Fubini's theorem gives \eqref{eq:hbs} is equal to
\begin{align}
    \tfrac{1-(a+b)}{2}  \sum_{k=0}^\infty p_{k} t^{k} \sum_{\ell=0}^{\infty} \hat \delta_{\ell} t^{\ell}  \sum_{n=0}^{\infty} p_{n}t^{n} = \tfrac{1-(a+b)}{2} \hat D(t) f(t)^2.
\end{align}    
 The other  derivations are similar. 

We have thus established that
\begin{align}
f(t) = p_0 + p_1 t + A(t) + B(t) + C(t)	\label{eq:f}
\end{align} 
with $p_0=0$ and $p_1 = (1-p)/2$.
Using the just-derived equations for $A,B,$ and $C$, the right side of \eqref{eq:f} can be expanded into an expression involving $f$ along with factors of $p,a,b,c,$ and $t$. The characterization is implicit since the formula for $\bar D(t)$ is also a recursive equation and may not necessarily have a solution. 
\end{proof}

\bibliographystyle{amsalpha}
\bibliography{BA.bib}

\end{document}